\documentclass[11pt,a4paper]{article} 
\usepackage{amsmath,amssymb,amsthm,amsfonts}
\usepackage{enumerate,color,bm}
\usepackage[utf8]{inputenc}
\usepackage[T1]{fontenc}
\usepackage{geometry}
\numberwithin{equation}{section} 
\newtheorem{thm}{Theorem}[section]

\newtheorem{lem}[thm]{Lemma}

\theoremstyle{definition}

\newtheorem{remark}[thm]{Remark}

\newcommand{\ep}{\varepsilon}

\newcommand{\f}[2]{\frac{#1}{#2}}
\newcommand{\na}{\nabla}

\newcommand{\Om}{\Omega}

\newcommand{\lp}[2]{\|#2\|_{L^{#1}(\Omega)}}
\newcommand{\lpb}[2]{\big\|#2\big\|_{L^{#1}(\Omega)}}

\newcommand{\ol}{\overline}

\usepackage[dvipsnames]{xcolor}

\newcommand{\cd}{(\cdot,t)}

\newcommand{\io}{\int_{\Omega}}

\usepackage{newunicodechar}
\newunicodechar{η}{\eta}
\newunicodechar{Δ}{\Delta}
\newunicodechar{χ}{\chi}
\newunicodechar{ε}{\varepsilon}
\newunicodechar{δ}{\delta}
\newunicodechar{ℝ}{\mathbb{R}}

\newcommand{\ubar}{\overline{u}}
\newcommand{\vstar}{v_*}
\newcommand{\mass}{M}
\newcommand{\wasM}{Z}
\newcommand{\etatilde}{\widetilde{\eta}}

\begin{document}
\begin{center}
    \LARGE{{\bf 
 Stabilization in the Keller--Segel system \\with signal-dependent sensitivity}}
\end{center}
\vspace{5pt}
\begin{center}
    Tobias Black\\
    \vspace{2pt}
    Universit\"at Paderborn, 
    Institut f\"ur Mathematik,\\ 
    Warburger Str.\ 100, 33098 Paderborn, Germany\\
    {\tt tblack@math.uni-paderborn.de}\\
    \vspace{12pt} 
    Johannes Lankeit\\
    \vspace{2pt}
    Universit\"at Paderborn, 
    Institut f\"ur Mathematik,\\ 
    Warburger Str.\ 100, 33098 Paderborn, Germany\\
    {\tt jlankeit@math.uni-paderborn.de}\\
    \vspace{12pt}
    Masaaki Mizukami\\
    \vspace{2pt}
    Department of Mathematics, 
    Tokyo University of Science\\
    1-3, Kagurazaka, Shinjuku-ku, Tokyo 162-8601, Japan\\
    {\tt masaaki.mizukami.math@gmail.com}\\
    \vspace{2pt}
\end{center}
\begin{center}    
    \small \today
\end{center}

\vspace{2pt}
\newenvironment{summary}
{\vspace{.5\baselineskip}\begin{list}{}{%
     \setlength{\baselineskip}{0.85\baselineskip}
     \setlength{\topsep}{0pt}
     \setlength{\leftmargin}{12mm}
     \setlength{\rightmargin}{12mm}
     \setlength{\listparindent}{0mm}
     \setlength{\itemindent}{\listparindent}
     \setlength{\parsep}{0pt}
     \item\relax}}{\end{list}\vspace{.5\baselineskip}}
\begin{summary}
{{\bf Abstract.}
This paper deals with the Keller--Segel system with signal-dependent sensitivity 
\begin{align*}
 &u_t =  \Delta u - \chi \na \cdot (uS(v)\na v), 
\\
 &v_t =  \Delta v - v + u,
\end{align*}
where $\chi>0$ and $S$ is a given function generalizing the sensitivity $S(s)=\frac{1}{(a+s)^{k}}$, $k>1$, $a\ge 0$,   
and shows exponential convergence of global classical solutions under an additional smallness condition condition for $\chi>0$. 
\vspace*{1cm}
}
\end{summary}
\vspace{10pt}

\newpage
%
%
\section{Introduction}

\subsection{Long-term behaviour in chemotaxis models}
Although mainly known for admitting solutions that blow up, the class of chemotaxis models also encompasses a large variety of systems whose solutions are global and bounded. In these situations the question of long-term behaviour of the solutions becomes significant. However, even in the most prototypical situation, 
\begin{equation}\label{KS}
 \begin{cases}
  u_t=Δu-χ\nabla\cdot(uS(v)\nabla v),\\
  v_t = Δv-v+u,
 \end{cases}
\end{equation}
with $S(v)\equiv1$ (and $χ>0$ so small that solutions in bounded two-dimensional domains are global), the answer to this question is not straightforward: Whereas each of the solutions converges to a stationary state, \cite{Feireisletal}, the set of those steady state solutions is rather non-trivial, \cite{biler,senba_suzuki,schaaf}. 
If $S(v)$ takes a different form---prototypical choices being $S(v)=\f1v$ or $S(v)=\f1{(1+v)^k}$, see also \cite[Sec. 2.2]{hillenpainter_survey}---, the system \eqref{KS} even loses the energy structure on which the proof of \cite{Feireisletal} is based. 

In these situations, even global existence is only known under additional restrictions: If $S(v)=\f1v$, a smallness condition on $χ$ (\cite{Biler99, Winkler_2011, Fujie_2015,Johannes_2016}), vastly different diffusion speeds (\cite{Fujie-Senba_2016,Fujie-Senba_2018}) or pursuance of a weaker solution concept (\cite{Winkler_2011,Stinner-Winkler_2011,Lankeit-Winkler_2017,zhigun}) have been needed for corresponding proofs. 

Nevertheless, recently, posing even stricter smallness conditions on $χ$, Winkler and Yokota in \cite{Winkler-Yokota} obtained global asymptotic stability of the homogeneous state $(\ubar_0,\ubar_0)$, thus highlighting the strong qualitative differences between the classical Keller--Segel model (where large perturbations of the stationary state usually lead to blow-up in finite time, see e.g. \cite{Winkler_2013_blowup}) and chemotaxis systems with logarithmic sensitivity ($S(v)\nabla v= \nabla \log v$). 

For similar sensitivities, global existence has been assured in the radial setting and if the chemical diffuses fast (\cite{Fujie-Senba_2016}), or, alternatively, whenever 
\[
\chi < k (a+\eta)^{k-1}\sqrt{\frac{2}{n}},
\]
if $S(v)\le \f1{(a+v)^k}$ with some $k>1$, $a\ge 0$, \cite{Mizukami-Yokota_02}, where $η>0$ has the form discussed below.

It is the latter case that we want to examine in regards to its asymptotic behaviour in the present paper. This goal compels us to revisit the boundedness proof of \cite{Mizukami-Yokota_02}, since now more quantitative information becomes necessary. Said bounds at hand, we can then begin adapting the reasoning of \cite{Winkler-Yokota} about the large-time asymptotics. 

In contrast to the situation there, not only smallness of the chemotaxis coefficient, but also, remarkably, largeness of the initial mass $\io u_0$ lead to eventual equilibriation. 
\medskip 

\subsection{Main results}
In order to state the results, let us first introduce the precise setting: 
We investigate the chemotaxis system with signal-dependent sensitivity 
 \begin{align}\label{cp}
     \begin{cases}
         u_t = \Delta u - \chi \na \cdot \big(uS(v)\na v\big)
         &\text{in } \Omega\times(0,T),   
 \\[1mm]
         v_t = \Delta v - v + u
         &\text{in } \Omega\times(0,T),  
 \\[1mm] 
        \partial_\nu u =
        \partial_\nu v = 0 
        & \text{in }\partial\Omega\times(0,T),
 \\[1mm] 
        u(\cdot,0)=u_0,\ v(\cdot,0)=v_0 
        &  \text{in } \Omega, 
    \end{cases}
 \end{align}
where $\Omega\subset \mathbb{R}^n$ ($n\ge 2$) is a bounded domain 
with smooth boundary $\partial \Omega$, $\partial_\nu$ denotes differentiation with respect to the outward normal of $\partial \Omega$, $\chi>0$ is a constant, $S$ is a given function and $u_0,v_0$ are also given initial data satisfying  

\begin{align}\label{condi;ini} 
0\le u_0 \in C^0(\ol{\Omega})\setminus \{0\}   
\quad \mbox{and} \quad 
 0\le v_0 \in W^{1,\infty}(\Omega)\setminus \{0\}. 
\end{align}

The main result of this article will then be given by: 

\begin{thm}\label{mainthm}
For $n \ge 2$ let $\Omega\subset \mathbb{R}^n$ 
 be a bounded domain with smooth boundary. 
Let $k>1$, $M>0$ and $v_\star>0$, $a\ge 0$. 
Then there is $δ=δ(a,k,M,v_\star)>0$ such that for all functions  
\begin{align}\label{condi;S} 
S\in C^{1+\theta}((0,\infty)) \text{ for some } \theta\in(0,1), \qquad \text{with}\qquad 
 0 \le S(s) \le \frac{1}{(a+s)^k}, 
 \end{align}
all initial data $u_0, v_0$ as in \eqref{condi;ini} and satisfying $\io u_0=M$ and $\min v_0=v_\star$ and for all $χ<δ$, the problem \eqref{cp} has a global classical solution 
\begin{equation}
\begin{aligned}\label{th;class}
&u\in C^{2,1}\big(\ol{\Omega}\times(0,\infty)\big)\cap C^0\big(\ol{\Omega}\times [0,\infty)\big), 
\\ 
&v\in C^{2,1}\big(\ol{\Omega}\times(0,\infty)\big)\cap C^0\big(\ol{\Omega}\times [0,\infty)\big)\cap L^\infty \big((0,\infty);W^{1,\infty}(\Omega)\big)
\end{aligned}
\end{equation}
and for this solution one can find $\kappa>0$ and $C>0$ such that 
\begin{equation}\label{result:convergenceestimate}
  \lp{\infty}{u\cd -\ol{u_0}}+ \lp{\infty}{v\cd -\ol{u_0}}\le Ce^{-\kappa t} 
\quad \mbox{for all }t>0. 
\end{equation}
\end{thm}

\begin{remark}\label{rem;deltaindependence}
In the case that $a>0$, $δ$ can be chosen independently of $M$ and $v_\star$.
\end{remark}

The strategy for the proof of this result lies in identifying the functional 
\[
 {\cal{F}}(u,v)(t) := \io (u\cd -\ol{u_0})^2 + K\io (v\cd -\ol{u_0})^2
\]
(for some $K>0$) as eventual Lyapunov functional (cf. \cite{Winkler-Yokota}). Apparent estimates for its dissipation rate show $L^2$-convergence, which then by means of boundedness information for the solutions can be ubgraded to $L^\infty$-convergence. 
One of the keys to construct an asymptotic estimate for the Lyapunov functional is to obtain an asymptotic universal estimate for $uS(v)$. This estimate will be the objective of Section \ref{sec;AE}, where we will employ the function 
\[
 \varphi(s) := \exp\left\{\frac{r}{(k-1)(a+s)^{k-1}} \right\}, 
 \qquad \qquad s>0,
\]
with some $r>0$, which is similar to that used in \cite{Mizukami-Yokota_02}, to establish the estimate as 
\[
 \limsup_{t\to \infty}\io u^p \varphi(v) \le C_1 \io u_0
\]
with some $C_1>0$. Smoothing estimates for the heat semigroup will enable us to turn this into
\[
 \limsup_{t\to \infty}\lp{\infty}{u\cd S(v\cd)} \le C_2
\] 
with some $C_2$. If $χC_2$ is sufficiently small, this facilitates the Lyapunov type arguments alluded to above. (They will be given in Section \ref{sec;conv}.) It turns out that, due to $k>1$ (cf. \eqref{ineq;uSv}), the ``constants'' $C_1$ and $C_2$ depend on $M=\io u_0$ in such a way that actually large initial masses augment the chances for convergence:

\begin{thm}\label{secondthm}
 Let $n\ge 2$ and let $\Om\subset ℝ^n$ be a bounded domain with smooth boundary, let $a\ge 0$, $k>1$, let $v_\star>0$ and $χ_0\in\big(0,k(a+v_\star)^{k-1}\sqrt{\f2n}\big)$. Then there is $M_0>0$ such that for every $M\ge M_0$, for every $χ\in(0,χ_0)$, every function $S$ as in \eqref{condi;S} and all initial data $u_0$, $v_0$ as in \eqref{condi;ini} that satisfy $\io u_0=M$ and $\min v_0\ge v_\star$, the problem \eqref{cp} has a global classical solution converging as indicated in \eqref{result:convergenceestimate}.
\end{thm}

The proofs of both theorems (and of Remark \ref{rem;deltaindependence}) will be given at the end of Section \ref{sec;conv}. 

{\bf Notation.} While constants $C_i$ are ``local'' to each proof, constants denoted by $K_i$ are meant to be the same ones as introduced with the same name by a previous Lemma.
 
\section{An asymptotic universal estimate for $\bm{uS(v)}$}\label{sec;AE}

Let us start with recalling some properties which have been established in previous studies and are fundamental when discussing results concerning the global existence of classical solutions in the setting of \eqref{cp}. 
For any $T\in(0,\infty]$ the arguments employed in the proof of \cite[Lemma 3.1]{Fujie-Senba_2018} show that all classical solutions $(u,v)$ of \eqref{cp} in $\Omega\times(0,T)$ satisfy the following time-independent lower estimate for $v$:
\[
 v(x,t) \ge \eta \quad \mbox{for all} \ x\in \Omega \ \mbox{and} \ t\in (0,T),
\]
where $\eta$ is defined by 
\begin{align}\label{def;eta;GL}
 \eta :=  4 \left(1+\sqrt{1+\f{4 \vstar}{c_0 \mass}}\right)^{-2}\vstar, \qquad \mass=\io u_0,\quad \vstar=\min v_0,
\end{align}
with $c_0>0$ being a lower bound for the fundamental solution of $w_t=\Delta w -w$ with Neumann boundary condition. The formula in \eqref{def;eta;GL} is an explicit form of the expression given in \cite[(1.5)]{Mizukami-Yokota_02}. 

In \cite[Theorem 1.1]{Mizukami-Yokota_02}, the above inequality was utilized to show existence of time-global classical solutions. 
We state it for reference in the next lemma.  

\begin{lem}\label{lem;GE}
Assume that $S\in C^{1+\theta}([0,\infty))$, with some $\theta>0$, satisfies \eqref{condi;S} for some $a\geq0$ and $k>1$ and that $u_0, v_0$ fulfill \eqref{condi;ini}. Let $\eta$ be given by \eqref{def;eta;GL} and suppose $\chi\in \big( 0,k(a+\eta)^{k-1}\sqrt{\frac 2n} \big)$. 
Then the problem \eqref{cp} possesses a 
unique global classical solution $(u,v)$ satisfying 
\eqref{th;class}, $u>0$ and $v>0$ in $\Omega\times (0,\infty)$, 
and 
\[
  u\in L^\infty \big(\Omega\times (0,\infty)\big) 
  \quad \mbox{and} \quad 
  v\in L^\infty \big(\Omega\times (0,\infty)\big)
\] 
and moreover, 
\[
  \io u(\cdot,t) = \io u_0 \quad \mbox{for all} \ t>0. 
\]
\end{lem}

In the case of $a=0$, in order to control $S(v)$ for large time, we will need to consider asymptotic upper bounds on $\frac{1}{v^k}$ which do not depend on time. Hence, we will make use of the following lower estimate for $v$ which only depends on $\Omega$ and the mass of $u_0$, which was established in \cite[Lemma 3.1]{Winkler-Yokota}.

\begin{lem}\label{lem;lower;v}
There exists $K_1>0$ such that 
whenever $(u,v)$ is a global classical solution of \eqref{cp} 
for some $\chi >0$ and some $(u_0,v_0)$ fulfilling \eqref{condi;ini}, 
the inequality 
\[
  \liminf_{t\to \infty}\inf_{x\in \Omega} v(x,t) \ge 2K_1 \io u_0 
\]
holds. 
\end{lem}

In order to prepare a testing procedure suitable to our purpose, we will now consider a test function $\varphi\in C^2((0,\infty))$ which has a structural resemblance to the test function used in \cite[Lemma 3.2]{Mizukami-Yokota_02}. To be precise, for $a\ge0$ and $k>1$ as in \eqref{condi;S} and some $r>0$ we define
\begin{align}\label{def;phi}
  \varphi(s) := \exp \left\{\frac{r}{(k-1)(a+s)^{k-1}}\right\} 
  \quad \mbox{for} \ s>0.
\end{align}
Obviously, from straightforward differentiation we find that 
\begin{align}\label{dif;phi}
 \varphi'(s) = -\frac{r}{(a+s)^k} \varphi(s) \quad \mbox{for all}\ s>0,
\end{align}
which will be used in the next lemma to derive a differential inequality for functionals of the form $\io u^p \varphi(v)$ with some $p >\frac{n}{2}$. 

\begin{lem}\label{lem;dif;intupphiq}
Let $r>0$, $a\ge 0$, $k>1$, $p>1$, $\ep\in(0,1)$, $χ>0$ 
 and let $\varphi$ be the function defined by \eqref{def;phi}. 
Then for every global classical solution $(u,v)$ to \eqref{cp}, the inequality 
\[
\frac  d{dt} \io u^p \varphi(v) 
\le  - \ep p(p-1) \io u^{p-2}\varphi(v)|\na u|^2 
  + \io u^p H_{\ep,p,r,\chi}(v) \varphi (v)|\na v|^2 
 + r\io u^p \varphi (v) \frac{v-u}{(a+v)^k}
\] 
holds on $(0,\infty)$, with 
\begin{align}\label{def;Hep}
H_{\ep,p,r,\chi}(s) :=  
-\frac{\chi pr S(s)}{(a+s)^k} 
    - \frac{r^2}{(a+s)^{2k}} -  \frac{kr}{(a+s)^{k+1}} 
+ \frac{\left(
   \frac{2pr}{(a+s)^k} +\chi p(p-1)S(s)
   \right)^2}{4(1-\ep)p(p-1)}\quad \text{for } s>0.
\end{align}
\end{lem}
\begin{proof}
From straightforward calculations, while relying on \eqref{dif;phi}, we derive that 
\begin{align}\label{start;dif}
 \frac d{dt} \io u^p\varphi(v) 
 &= p\io u^{p-1}\varphi(v) \big( \Delta u -\chi \na \cdot (uS(v)\na v)\big) 
  - r\io u^{p} \frac{\varphi (v)}{(a+v)^k} (\Delta v - v +u)
\end{align}
on $(0,\infty)$. 
Here, noting from integration by parts and \eqref{dif;phi} that  
\begin{align*} 
p \io u^{p-1} \varphi(v)\Delta u 
= - p(p-1) \io u^{p-2}\varphi(v)|\na u|^2
+pr \io u^{p-1}\frac{\varphi(v)}{(a+v)^k}\na u \cdot \na v
\end{align*}
and 
\begin{align*}
-\chi p \io u^{p-1}\varphi(v) \na \cdot (uS(v)\na v)
& = \chi p(p-1) \io u^{p-1} S(v) \varphi(v)\na u\cdot \na v 
\\ &\quad  \, 
-\chi pr \io u^p S(v) \frac{\varphi (v)}{(a+v)^k} |\na v|^2
\end{align*}
as well as 
\begin{align*}
 - r \io u^p\frac{\varphi(v)}{(a+v)^k} \Delta v 
 &=  rp\io u^{p-1} \frac{\varphi(v)}{(a+v)^k} \na u\cdot \na v 
 - r^2 \io u^p \frac{\varphi(v)}{(a+v)^{2k}} |\na v|^2 
\\
 &\quad \,  - kr \io u^p\frac{\varphi(v)}{(a+v)^{k+1}} |\na v|^2 
\end{align*}
hold on $(0,\infty)$, we obtain from \eqref{start;dif} that 
\begin{align*}
&\frac  d{dt} \io u^p \varphi(v) 
\\
&= 
-p(p-1) \io u^{p-2}\varphi(v)|\na u|^2
+ \io u^{p-1}\varphi(v)
  \left(
   \frac{2pr}{(a+v)^k} +\chi p(p-1)S(v)
   \right) 
  \na u \cdot \na v
\\
& \quad \, + \io u^p 
  \left( 
    -\frac{\chi pr S(v)}{(a+v)^k} 
    - \frac{r^2}{(a+v)^{2k}} -  \frac{kr}{(a+v)^{k+1}}
  \right) 
  \varphi (v)|\na v|^2
+ r\io u^p \varphi (v) \frac{v-u}{(a+v)^k}. 
\end{align*}
Now we let $\ep\in (0,1)$. 
Then from Young's inequality we have
\begin{align*}
u^{p-1}\varphi(v)
  \left(
   \frac{2pr}{(a+v)^k} +\chi p(p-1)S(v)
   \right) 
  \na u \cdot \na v 
&\le 
  (1-\ep) p(p-1) u^{p-2} \varphi(v) |\na u|^2 
\\ &\quad\, 
  + \frac{\left(
   \frac{2pr}{(a+v)^k} +\chi p(p-1)S(v)
   \right)^2}{4(1-\ep)p(p-1)} u^p \varphi(v)|\na v|^2  
\end{align*}
and infer that 
\begin{align*}
\frac  d{dt} \io u^p \varphi(v) 
&\le  - \ep p(p-1)\! \io u^{p-2}\varphi(v)|\na u|^2 
  +\! \io u^p H_{\ep,p,r,\chi}(v) \varphi (v)|\na v|^2 
 +\! r\io u^p \varphi (v) \frac{v-u}{(a+v)^k},  
\end{align*}
where $H_{\ep,p,r,\chi}$ is the function defined by \eqref{def;Hep}, 
which completes the proof. 
\end{proof} 

Observing that the differential inequality only depends on $\chi$ inside the function $H_{\ep,p,r,\chi}$, we can conclude that whenever the sign of $H_{\ep,p,r,\chi}$ is non-positive the chemotactic influence in this inequality is negligible. Our next aim is to verify that one can find a suitable combination of parameters $\ep,p,r$ and $\chi_0$ such that $H_{\ep,p,r,\chi}$ is bounded from above by zero independently of $\chi\in(0,\chi_0]$.

\begin{lem}\label{lem;H<=0}
Let $a\ge 0$, $k>1$ and $\etatilde>0$. 
For all $\chi_0 \in \big( 0,k(a+\etatilde)^{k-1}\sqrt{\frac 2n} \big)$ 
there exist $p=p(\chi_0,a,k,\etatilde)>\frac n2$ 
and $\ep=\ep(\chi_0,a,k,\etatilde) \in (0,1)$  such that 
\[
  H_{\ep,p,r,\chi}(s) \le 0 \quad \mbox{for all} \ s \ge \etatilde 
\]
holds for all $\chi \in (0,\chi_0]$ with 
\begin{align}\label{def;r}
 r:= \frac{(p-1)\chi_0}{2}\sqrt{\frac p{1+\ep p -\ep}}.  
\end{align} 
\end{lem}
\begin{proof}
Since $\chi_0 < k(a+\etatilde)^{k-1}\sqrt{\frac 2n}$, 
we can find $p=p(\chi_0,a,k,\etatilde)\in (\frac n2,n)$ and 
$\ep=\ep(\chi_0,a.k,\etatilde)\in (0,1)$ such that 
\[
 \frac{\chi_0}{k(1-\ep)}\sqrt{p(1+\ep p -\ep)} + \ep p \chi_0 
 \le (a+ \etatilde)^{k-1}. 
\]
Now we let $r$ be given as in \eqref{def;r}. 
Then straightforward calculations using condition \eqref{condi;S} and the fact $\chi \le \chi_0$ imply that 
\begin{align*}
 H_{\ep,p,r,\chi}(s)& = 
 \frac{(1+\ep p - \ep)}{(1-\ep)(p-1)(a+s)^{2k}}r^2 
 +\frac{\ep p \chi rS(s)}{(1-\ep)(a+s)^{k}}
 + \frac{p(p-1)\chi^2 S^2(s)}{4(1-\ep)} 
 - \frac{kr}{(a+s)^{k+1}}
\\ 
 & \le \frac{kr}{(a+s)^{2k}}
 \left(
   \frac{(1+\ep p - \ep)}{k(1-\ep)(p-1)}r 
   +\frac{\ep p \chi_0}{k(1-\ep)}
 + \frac{p(p-1)\chi_0^2}{4k(1- \ep)r} 
 - (a+s)^{k-1}
 \right).
\end{align*} 
Here, noting from the definition of $r$ and $p,\ep$ that 
\begin{align*}  
 \frac{(1+\ep p - \ep)}{k(1-\ep)(p-1)}r 
 +\frac{\ep p \chi_0}{k(1-\ep)}
 + \frac{p(p-1)\chi_0^2}{4k(1- \ep)r} 
 = \frac{\chi_0}{k(1-\ep)}\sqrt{p(1+\ep p -\ep)} + \ep p \chi_0 
 \le (a+\etatilde)^{k-1},  
\end{align*} 
we can verify that 
\begin{align*}
 H_{\ep,p,r,\chi} (s) 
 \le 
 \frac{kr( (a+\etatilde)^{k-1}-(a+s)^{k-1})}{(a+s)^{2k}} 
 \le 0 \quad \mbox{for all} \ s \ge \etatilde  
\end{align*}
is valid. 
\end{proof}

Combining the Lemmata \ref{lem;dif;intupphiq} and \ref{lem;H<=0}, we can now derive the following asymptotic $L^p$-estimate of the first solution component for a certain choice of $p>\frac{n}{2}$.

\begin{lem}\label{lem;esti;u^p}
Let $M_0\ge0$, $a\ge 0$, $k>1$ and suppose that $S$ satisfies \eqref{condi;S}. Then
for all $\chi_0 \in (0,k(a+K_1M_0)^{k-1}\sqrt{\frac 2n})$ 
there exist $p=p(\chi_0,a,k,M_0)>\frac{n}{2}$ and 
$K_2 = K_2 (\chi_0,a,k,M_0)$ such that 
whenever $(u,v)$ is a global classical solution of \eqref{cp} 
with $\chi \le \chi_0$ and $(u_0,v_0)$ fulfilling \eqref{condi;ini} as well as $\io u_0=\mass\ge M_0$, 
we have 
\begin{align}\label{ineq;limsup;u^p} 
 \limsup_{t\to \infty} \lp{p}{u\cd} \le K_2 \io u_0.  
\end{align}
\end{lem} 
\begin{proof}
We first note that, aided by Lemma \ref{lem;lower;v}, 
we can find $t_0 >0$ such that 
\[
 v(x,t) \ge K_1\mass\ge K_1M_0 \quad 
  \mbox{for all} \ x\in  \Omega \ \mbox{and all} \ t>t_0. 
\]
Then combining Lemma \ref{lem;dif;intupphiq} and Lemma \ref{lem;H<=0} and choosing $p=p(χ_0,a,k,K_1M_0)$ and $ε=ε(χ_0,a,k,K_1M_0)$ as in the latter,  
we derive that with $\varphi$ as in \eqref{def;phi} 
\begin{align*}
 \frac d{dt} \io u^p \varphi(v) 
 \le 
 - \ep p (p-1) \io u^{p-2} \varphi (v) |\na u|^2 
 + r\io u^p \varphi (v) \frac{v-u}{(a+v)^k}
\end{align*}
holds for all $t>t_0$. 
From positivity of $u$ and the definition of $r>0$ (in \eqref{def;r}) we obtain that
\begin{align*}
\frac d{dt} \io u^p \varphi(v) 
 \le 
 - \ep p (p-1) \io u^{p-2} \varphi (v) |\na u|^2 
 + C_1 \io u^p \varphi(v)
\end{align*}
for all $t>t_1$ with $C_1:=\frac{(p-1)\chi_0}{2}\sqrt{\frac{p}{1+\ep p -\ep}} \f{m_*}{(a+m_*)^k}$, where $m_*:=\max\{\f{a}{k-1},K_1M_0\}$. 
Here we note that 
\begin{align}\label{ineq;phi}
  C_\varphi := 
  \exp\left\{ 
    -\frac{(p-1)\chi_0}{2(k-1)(a+K_1M_0)^{k-1}}
    \sqrt{\frac{p}{1+\ep p -\ep}} 
  \right\} 
  \le \varphi(s) \le 1 
  \quad \mbox{for all} \ s\ge K_1M_0.    
\end{align}
Thanks to the upper estimate in \eqref{ineq;phi}, 
the Gagliardo--Nirenberg inequality and 
the mass conservation law entail that on $(t_0,\infty)$ 
\begin{align*}
 \io u^p \varphi (v) \le \io u^p 
 = \lp{2}{u^{\frac p2}}^2 
 &\le C_{GN} \Big(\lp{2}{\na u^{\frac p2}}^2+ \lp{\frac 2p}{u^{\frac p2}}^2\Big)^{b} 
 \lp{\frac 2p}{u^{\frac p2}}^{2(1-b)}  
\\
&= 
 C_{GN} \left(
   \io |\na u^{\frac p2}|^2+ \mass^p
\right)^{b} 
 \mass^{p(1-b)} 
\end{align*}
holds with $b:= \frac{(p-1)n}{(p-1)n+2}\in (0,1)$ and some $C_{GN}>0$,  
which means that 
\begin{align*}
 \io u^{p-2} \varphi (v) |\na u|^2 
 \ge \frac{4C_\varphi}{p^2}\io |\na u^{\frac p2}| 
 \ge \frac{4C_\varphi}{p^2}\left( 
  \left( C_{GN} \mass^{p(1-b)}\right)^{-\frac 1b}\left(\io u^p\varphi(v)   \right)^{\frac 1b} -\mass^p 
 \right)
\end{align*}
for all $t>t_0$. 
Therefore we have from Young's inequality that 
\begin{align*}
 \frac d{dt} \io u^p\varphi(v) 
 &\le -C_2 (C_{GN} \mass^{p(1-b)})^{-\frac 1b} 
  \left(\io u^p\varphi(v)\right)^{\frac 1b} + C_2 \mass^p + C_1 \io u^p\varphi(v) 
\\
 &\le -\frac{C_2 (C_{GN} \mass^{p(1-b)})^{-\frac 1b}}{2} 
   \left(\io u^p\varphi(v)\right)^{\frac 1b}
   + C_3 \mass^p 
\end{align*}
for all $t>t_0$, where 
$C_2:= \frac{4\ep (p-1)C_\varphi}{p}$ and 
$C_3:= C_2 + (1-b)(C_1 C_{GN})^{\frac 1{1-b}}(\frac{2b}{C_2})^{\frac b{1-b}}$, 
which with \eqref{ineq;phi} implies that 
\begin{align*}
  \limsup_{t\to \infty} \lp{p}{u\cd} 
  &\le \frac{1}{C_\varphi^{\frac 1p}}
  \left( \limsup_{t\to \infty}  \io u^p\varphi(v) \right)^{\frac{1}{p}}
\\
  &\le \frac{1}{C_\varphi^{\frac 1p}}\left( \frac{2C_3\mass^p}{C_2} 
  \left( C_{GN}\mass^{p(1-b)} \right)^{\frac 1b} \right)^{\frac{b}{p}}
  = 
  \left(\frac{C_{GN}}{C_\varphi}\right)^{\frac{1}{p}} 
  \left( \frac{2C_3}{C_2}\right)^{\frac{b}{p}}\mass
\end{align*} 
because $\frac 1b>1$. Thus, \eqref{ineq;limsup;u^p} holds with 
$K_2:= (\frac{C_{GN}}{C_\varphi})^{\frac{1}{p}} 
  (\frac{2C_3}{C_2})^{\frac{b}{p}}$. 
\end{proof}

Still striving for an asymptotic $L^\infty$-estimate for $uS(v)$, we nevertheless need to obtain additional regularity information on $\nabla v$, since when estimating $u$, we lack control on the crucial term $uS(v)\nabla v$ with our current knowledge. In particular, an $L^{q_0}$-estimate for some $q_0>n$ would suffice for our purpose. Fortunately, the regularity of $\nabla v$ is directly linked to the $L^p$-regularity of $u$, as illustrated by the following result (cf. \cite[Lemma 3.2]{Winkler-Yokota}).

\begin{lem}\label{lem;esti;v}
Let $\mu\ge 1$ and $\lambda \ge 1$ be such that $\lambda < \frac {n\mu}{(n-\mu)_+}$. 
Then there is $K_3 = K_3(\mu,\lambda) >0$ such that whenever 
$(u,v)$ is a global classical solution of \eqref{cp}, for any $S$, $u_0$, $v_0$ as in \eqref{condi;S} and \eqref{condi;ini}, respectively, 
then the inequality 
\begin{align*}
  \limsup_{t\to \infty} \|v\cd \|_{W^{1,\lambda}(\Omega)} 
  \le K_3 \limsup_{t\to\infty} \lp{\mu}{u\cd} 
\end{align*}
holds.
\end{lem} 

In light of this result and Lemma \ref{lem;esti;u^p} we can now draw on quite standard smoothing properties of the Neumann heat-semigroup to derive an asymptotic $L^\infty$-estimate for $u$. 

\begin{lem}\label{lem;esti;fracu}
Let $\sigma \in (0,\lambda_1)$, where $\lambda_1>0$ denotes the first nonzero eigenvalue of the Neumann Laplacian in $\Omega$. Let $a\ge 0$, $k>1$ and $M_0\ge0$. 
For all $\chi_0 \in \big( 0,k(a+K_1M_0)^{k-1}\sqrt{\frac 2n} \big)$ 
there are $\theta = \theta (\chi_0,a,k,M_0)>n$ and $\alpha = \alpha (\chi_0,a,k,M_0) <1$ as well as $K_4>0$ and $K_5>0$ such that the following holds:
If $(u,v)$ is a global classical solution of \eqref{cp} with 
$\chi\le \chi_0$ and $(u_0,v_0)$ satisfying \eqref{condi;ini} and $\io u_0=\mass\ge M_0$, 
then 
\begin{align*}
 \limsup_{t\to \infty} \lp{\theta}{A^\alpha u \cd} 
 \le K_4 \io u_0,   
\end{align*}
where $A$ denotes the sectorial realization of $-\Delta + \sigma$ 
in $L^\theta (\Omega)$ under homogeneous Neumann boundary conditions;
moreover, 
\begin{align*}
\limsup_{t\to \infty} \lp{\infty}{u \cd}\le \f{K_5}{2}\io u_0.
\end{align*}
\end{lem}
\begin{proof}
From Lemmata \ref{lem;lower;v}, \ref{lem;esti;u^p} and \ref{lem;esti;v} 
we can find $p>\frac n2$, $q\in (n, \frac{np}{(n-p)_+})$ and 
$t_0>0$ such that 
\begin{align}\label{p10-vlower}
 v(x,t) \ge K_1\io u_0  
\quad 
 \mbox{for all} \ 
 x\in \Omega \ \mbox{and all} \ t>t_0 
\end{align}
and such that 
\begin{align}\label{p10estimates2}
 \lp{p}{u\cd } \le 2K_2 \io u_0, 
 \quad \|v\cd\|_{W^{1,q}(\Omega)} \le 2K_2K_3 \io u_0
 \quad \mbox{for all} \ t>t_0  
\end{align}
hold. Now for $\theta \in (n,q)$ we let $\alpha \in \big( \frac{n}{2\theta},\min\{1-\frac n2(\frac 1p -\frac 1\theta), \frac 12\} \big)$ and put 
$\beta := \alpha + \frac n2 (\frac 1p - \frac 1\theta)<1$, 
$\gamma := \frac 12 + \alpha <1$, and moreover, 
for $T> t_0$ we define
\[
 \wasM (T) := \sup_{t\in (t_0,T]}\left(1+ (t-t_0)^{-\beta}\right)^{-1} 
 \lp{\theta}{A^\alpha u\cd}. 
\] 
Aided by the variation-of-constants representation of $u$, 
we have that, due to \eqref{p10-vlower} and \eqref{p10estimates2},
\begin{align*}
 \lp{\theta}{A^\alpha u\cd} 
 \le 
 \left\|A^\alpha e^{(t-t_0)\Delta}
 u(\cdot,t_0) \right\|_{L^\theta(\Omega)} 
 + \chi \int_{t_0}^t 
   \left\| A^\alpha e^{(t-s)\Delta}\na\cdot 
   (uS(v)\na v) (\cdot,s) \right\|_{L^\theta(\Omega)}\, ds 
\end{align*}
holds for all $t>t_0$. 
Here we note from the continuous embedding $D(A^\alpha) \hookrightarrow L^\infty (\Omega)$ (see \cite[Theorem 1.6.1]{Henry_1981}) 
\begin{align}\label{ineq;fracandinfty} 
 \lp{\infty}{\varphi} \le C_E \lp{\theta}{A^\alpha \varphi}
 \quad \mbox{for all} \ \varphi\in D(A^\alpha) 
\end{align}
with some $C_E>0$ that for all $s>t_0$
\begin{align*}
 \lpb{\theta}{(uS(v)\na v) (\cdot,s)}
 &\le 
 \frac{1}{(a+K_1\io u_0)^k}
 \lp{1}{u(\cdot,s)}^{1-c}
 \lp{\infty}{u(\cdot,s)}^{c} \lp{q}{\na v(\cdot,s)}
\\
 &\le 
 \frac{2C_E K_2 K_3 \io u_0}{(a+K_1\io u_0)^k} 
 \left(\io u_0\right)^{1-c}\wasM^c(T) 
 \left( 1+(s-t_0)^{-\beta} \right)^c
\\
 & \le 2C_E K_2 K_3 \f{M^{2-c}}{(a+K_1M)^k} \wasM^c(T) \left( 1+(s-t_0)^{-\beta} \right)^c
\end{align*} 
with $c:=1-\frac{q-\theta}{q\theta}\in (0,1)$. 

Noting from our choice of $\sigma$ and the known smoothing properties of the Neumann heat semigroup (see \cite[Theorem 1.4.3]{Henry_1981} and \cite[Lemma 1.3 (iv)]{win_aggregationvs}) that 
\begin{align*}
 \left\|A^\alpha e^{(t-t_0)\Delta}u(\cdot,t_0) \right\|_{L^\theta(\Omega)} 
 &\le C_{S_1} \left( 1+ (t-t_0)^{-\beta} \right) \lp{p}{u(\cdot,t_0)} 
 \\
 &\le 2K_2 C_{S_1} \left( 1+ (t-t_0)^{-\beta} \right)\io u_0  
\end{align*}
and 
\begin{align*}
 &\chi  \left\| A^\alpha e^{(t-s)\Delta}\na\cdot 
 ( uS(v)\na v)(\cdot,s) \right\|_{L^\theta(\Omega)}
 \\ 
 &\qquad \le 
  C_{S_2} \chi_0 (t-s)^{-\gamma} e^{-\lambda (t-s)} 
  \lp{\theta}{(uS(v)\na v) (\cdot,s)} 
\\  
 &\qquad \le  C_1
 \f{M^{2-c}}{(a+K_1M)^k} \wasM^c(T) 
(t-s)^{-\gamma} e^{-\lambda (t-s)} 
 \left( 1+(t-t_0)^{-\beta} \right)^c
\end{align*}
with some $C_{S_1},C_{S_2},\lambda>0$ and $C_1:=2C_E C_{S_2} K_2 K_3 \chi_0$,  
we have from the inequality 
\[
 \int_{t_0}^t (t-s)^{-\gamma}e^{-\lambda (t-s)}
 \left( 1+(s-t_0)^{-\beta} \right)^c\, ds 
 \le L \left( 1+ (t-t_0)^{-\beta} \right) \quad \mbox{for all} \ t> t_0 
\]
with some $L=L(\beta,\gamma,\lambda,c)>0$, obtained in \cite[Lemma 3.5]{Winkler-Yokota}, that
\begin{align*}
  \left(1+(t-t_0)^{-\beta}\right)^{-1}\lp{\theta}{A^\alpha u\cd} \le 
   2K_2 C_{S_1} \io u_0 +  C_1 L \f{M^{2-c}}{(a+K_1M)^k} \wasM^c(T). 
\end{align*} 
This together with Young's inequality 
\[
 C_1L \left(\io u_0\right)^{1-c} \wasM^c(T) \le 
 (1-c)(2c)^{\frac{c}{1-c}} \left(C_1L \f{M^{2-c}}{(a+K_1M)^k}\right)^{\frac 1{1-c}} + 
 \frac{1}{2} \wasM(T)  
\]
enables us to see that 
\[
 \wasM(T) \le C_2 
\left(
1+\left(\f{M}{(a+K_1M)^k}\right)^{\f{1}{1-c}}
\right)
\io u_0 \quad \mbox{for all} \ T> t_0
\]
with $C_2:= 2\max\{2 K_2 C_{S_1}, (1-c)(2c)^{\frac{c}{1-c}} (C_1L)^{\frac 1{1-c}})\}$. 
Therefore we attain that 
\begin{align*}
  \lp{\theta}{A^\alpha u\cd} 
  &\le \left(1+(t-t_0)^{-\beta}\right) \wasM(T) \\
  &\le 
C_3\left(1+\left(\f{M}{(a+K_1M)^{k}}\right)^{\f1{1-c}}\right)
 \io u_0 \quad \mbox{for all} \ t> t_0+1 
\end{align*}
with $C_3 := 2C_2$. Here, we finally set $K_4:=C_3(\f{m_*}{(a+K_1m_*)^k})^{\f{1}{1-c}}$, where $m_*=\max\{\f{a}{K_1(k-1)},M_0\}$. 
 To verify the second assertion, we make use of the first part of the lemma, to find that there exists some $t_1>0$ such that 
\begin{align*}
\lp{\theta}{A^\alpha u\cd}\leq 2K_4\io u_0
\end{align*}
is valid for all $t>t_1$. Then, we employ \eqref{ineq;fracandinfty} to find that
\begin{align*}
\lp{\infty}{u\cd}\leq 2C_E K_4 \io u_0\quad\text{for all}\ t>t_1
\end{align*}
which, by choice of $K_5:=4C_E K_4$, completes the proof. 
\end{proof}

We can now establish an asymptotic universal bound on $uS(v)$, which will be a key point in the proof of Theorem \ref{mainthm}.  

\begin{lem}\label{lem;esti;uSv}
Let $a\ge 0$, $k>1$ and $M_0\ge0$. 
For all $\chi_0 \in \big(0,k (a+K_1M_0)^{k-1}\sqrt{\frac 2n}\big)$ and with $K_5=K_5(\chi_0,a,k,M_0)$ as in the previous lemma, the following holds:
Whenever $(u,v)$ is a global classical solution of \eqref{cp} 
with $\chi\le \chi_0$ and $(u_0,v_0)$ satisfying \eqref{condi;ini} and $\io u_0=\mass\ge M_0$, 
we have 
\begin{align}\label{ineq;uSv}
 \limsup_{t\to \infty} \lp{\infty}{u\cd S(v\cd )} \le K_5 \f{M}{(a+K_1M)^k}.
\end{align} 
\end{lem}
\begin{proof}
Thanks to Lemmata \ref{lem;lower;v} and \ref{lem;esti;fracu}, 
we can find $t_0>0$ such that 
\begin{align*}
v(x,t) \ge K_1 \io u_0 
\quad \mbox{for all} \ x\in \Omega, \ t>t_0
\end{align*} 
and such that
\begin{align*}
\lp{\infty}{u\cd} \le K_5\io u_0 
\quad \mbox{for all} \ t> t_0. 
\end{align*}
Hence, we immediately obtain that for all $t>t_0$, 
\begin{equation*}
 \lp{\infty}{u\cd S(v\cd )} 
 \le  
 \frac{K_5\io u_0}{(a+K_1\io u_0)^k}. \qedhere 
\end{equation*}
\end{proof}

\section{An asymptotic estimate for the Lyapunov functional}\label{sec;conv}
As the existence part of the main theorem is covered by the previous section, we will now turn our attention to verifying the desired convergence result. Inspired by the approach undertaken in \cite{Winkler-Yokota}, we will consider the functional
\begin{align}\label{LyapFunc}
{\cal{F}}(u,v)(t):= \io (u (\cdot,t) - \ol{u_0})^2 + K \io (v(\cdot,t)-\ol{u_0})^2
\end{align}
with some $K>0$, which, at least for later times, acts as a Lyapunov functional to the system under consideration. We start by establishing a differential inequality for the first solution component. 

\begin{lem}\label{lem;asy;energy;u}
Let $a\ge 0$, $k>1$, $M_0\ge0$ and $\chi_0 \in \big(0,k(a+K_1M_0)^{k-1}\sqrt{\frac 2n}\big)$. 
Then every global classical solution $(u,v)$ of \eqref{cp} with some $S$ as in \eqref{condi;S}, 
$\chi \le \chi_0$ and $(u_0,v_0)$ fulfilling \eqref{condi;ini} and $\io u_0=\mass\ge M_0$  
satisfies 
\begin{align*}
 \frac d{dt} \io (u-\ol{u_0})^2 + \io |\na u|^2 
 \le 4K_5^2 \f{M^2}{(a+K_1M)^{2k}} \chi^2 \io |\na v|^2 
 \quad \text{ on } (t^\ast,\infty)
\end{align*}  
for some $t^\ast >0$, with $K_5=K_5(\chi_0,a,k,M_0)$ provided by Lemma \ref{lem;esti;fracu}. 
\end{lem} 
\begin{proof} 
From Lemma \ref{lem;esti;uSv} there is $t_0>0$ such that 
\begin{align*}
 u\cd S(v\cd ) \le 2 K_5 \f{M}{(a+K_1M)^k}
 \quad \mbox{for all} \ t>t_0.  
\end{align*}
Then testing the first equation of \eqref{cp} by 
$\frac 12 (u-u_0)$  and using integration by parts 
show  
\begin{align}\label{dif;energy;u}
 \frac 12 \frac d{dt} \io (u-u_0)^2 
 = -\io |\na u|^2 
   + \chi \io uS(v)\na u \cdot \na v 
   \quad \mbox{for all} \ t>t_0.  
\end{align}
Here we use Young's inequality to see that 
\begin{align}\notag \label{dif;energy;u;2}
 \chi \io uS(v) \na u \cdot \na v 
 &\le \frac 12 \io |\na u|^2 + \frac{\chi^2}{2} \io |uS(v)|^2|\na v|^2
\\
 & \le \frac 12 \io |\na u|^2 + 2\chi^2K_5^2 \f{M^2}{(a+K_1M)^{2k}} \io |\na v|^2 
\end{align}
ia valid for all $t>t_0$. 
Therefore a combination of \eqref{dif;energy;u} and 
\eqref{dif;energy;u;2} directly implies this lemma. 
\end{proof}

In the next lemma we will investigate the time-evolution of the second part of the Lyapunov-functional. 

\begin{lem}\label{lem;asy;energy;v} 
Every global classical solution $(u,v)$ of \eqref{cp}, for any choice of initial data permitted by \eqref{condi;ini} and $S$ as in \eqref{condi;S}, satisfies 
\begin{align*}
 \frac d{dt} \io (v-\ol{u_0})^2 
 + 2\io |\na v|^2 + \io (v-\ol{u_0})^2 
 \le \io (u-\ol{u_0})^2 
 \quad \mbox{for all} \ t>0. 
\end{align*}
\end{lem}
\begin{proof}
 Testing the second equation of \eqref{cp} by $v-\ol{u_0}$, 
 we have from the Young's inequality that 
 \begin{align*}
  \frac 12\frac d{dt} \io (v-\ol{u_0})^2 
  &= -\io |\na v|^2 - \io (v-\ol{u_0})^2 + \io (u-\ol{u_0})(v-\ol{u_0})
\\
  &\le -\io |\na v|^2 -\frac 12 \io (v-\ol{u_0})^2 + \frac 12 \io (u-\ol{u_0})^2
 \end{align*}
holds for all $t>0$. 
\end{proof}

Combining the previous two lemmata, for suitable choice of $K>0$, we can make use of \eqref{LyapFunc} to obtain convergence of solutions towards the spatial mean of $u_0$ in $L^2(\Omega)$ with an exponential rate.

\begin{lem}\label{lem;asy;energy}
Let $a\ge 0$, $k>1$, $M_0\ge 0$ and $\chi_0\in \big(0,k(a+K_1M_0)^{k-1}\sqrt{\frac 2n}\big)$. 
 Then there exists $\delta=\delta (\chi_0,a,k,M_0) >0$ such that, if $\chi\le \chi_0$ and $M\ge M_0$ satisfy 
 \[ \f{M}{(a+K_1M)^k}\chi < \delta,\]
then for every global solution $(u,v)$ of \eqref{cp} with $S$, $u_0$, $v_0$ as in \eqref{condi;S} and \eqref{condi;ini} and with $\io u_0=M$ there are $K_6, \ell>0$ and $t^\ast>0$ such that 
\begin{align}\label{ineq;L2;conv}
 \lp{2}{u(\cdot,t)-\ol{u_0}} + \lp{2}{v(\cdot,t)-\ol{u_0}} \le K_6 e^{-\ell t}
 \quad \mbox{for all} \ t> t^\ast. 
\end{align}
\end{lem}
\begin{proof} 
We first note from Poincar\'{e}'s inequality that 
there is $C_P >0$ such that 
\begin{align*}
 \io (\varphi-\ol{\varphi})^2 \le C_P \io |\na \varphi|^2
 \quad \mbox{for all} \ \varphi \in W^{1,2}(\Omega). 
\end{align*}
We put 
\[
  \delta :=\frac 1{K_5} \sqrt{\frac 1{2C_p}} 
\]
and assume that $\f{M}{(a+K_1M)^k} \chi < \delta$. Then we can choose $K>0$ such that 
\begin{align*}
  K\in \Big[ 2K_5^2\f{M^2}{(a+K_1M)^{2k}} \chi^2, \frac{1}{C_P}\Big). 
\end{align*}
Combination of Lemmata \ref{lem;asy;energy;u} and \ref{lem;asy;energy;v} 
entails that 
\begin{align*}
  \frac d{dt} \left( \io (u-\ol{u_0})^2 + K \io (v-\ol{u_0})^2 \right) 
  &+ \frac{1}{C_P} \io (u-\ol{u_0})^2 
  + 2K \io |\na v|^2 + K \io (v-\ol{u_0})^2 
\\
 &\le   4K_5^2 \f{M^2}{(a+K_1M)^{2k}} \chi^2 \io |\na v|^2 + K \io (u-\ol{u_0})^2
\end{align*}
holds for all $t> t_0$, 
where we utilized that 
\[
 \io (u-\ol{u_0})^2  = \io (u-\ol{u})^2 \le C_P \io |\na u|^2 
 \quad \mbox{for all} \ t>0. 
\]
Then, aided by the definition of $K>0$, 
we have that 
\begin{align*}
 \frac d{dt} \left( \io (u-\ol{u_0})^2 + K \io (v-\ol{u_0})^2 \right) 
 + C_1 \left( \io (u-\ol{u_0})^2 + K \io (v-\ol{u_0})^2 \right) \le 0 
\end{align*}
with $C_1 := \min\{ 1, \frac 1{C_P} - K \}$. 
This means that \eqref{ineq;L2;conv} holds with some $K_6>0$ and 
some $\ell>0$. 
\end{proof}

In fact, drawing on the bounds provided by Lemmata \ref{lem;esti;v} and \ref{lem;esti;fracu}, we can refine the exponential convergence in $L^2(\Omega)$ to an exponential convergence in $L^\infty(\Omega)$. 

\begin{lem}\label{lem;conv;Linfty}
Under the assumptions of Lemma \ref{lem;asy;energy}, 
 there exist $K_7>0$, $\kappa >0$ and $t_\star>0$ such that 
 \begin{align}\label{esti;exponential}
  \lp{\infty}{u-\ol{u_0}}
  + \lp{\infty}{v-\ol{u_0}} 
  \le K_7 e^{-\kappa t} 
  \quad \mbox{for all} \ t> t_\star. 
 \end{align} 
\end{lem} 
\begin{proof}
The proof is based on the arguments in \cite[Proof of Theorem 1.1]{Winkler-Yokota}. 
Let $p \in (\frac n2,n)$ be as in Lemma \ref{lem;esti;u^p} 
and let $q> n $ be such that $q<\frac{np}{n-p}$. 
In light of Lemmata \ref{lem;esti;u^p} and \ref{lem;esti;v}, 
we can find $t_0 > 0$ such that 
\begin{align*}
  \|v\cd \|_{W^{1,q}(\Omega)} \le 2 K_2 K_3 \io u_0 
  \quad \mbox{for all} \ t> t_0. 
\end{align*}
Then the Gagliardo--Nirenberg inequality 
enables us to see that 
\begin{align*}
 \lp{\infty}{v\cd - \ol{u_0}} 
 &\le C_1 \|v\cd -\ol{u_0} \|_{W^{1,q}(\Omega)}^{d_1} 
     \lp{2}{v\cd -\ol{u_0} }^{1-d_1}
\\ 
 &\le C_1\left( 2K_2 K_3 + |\Omega|^{\frac 1q}\ol{u_0} \right)^{d_1} 
 \lp{2}{v\cd -\ol{u_0} }^{1-d_1}
\end{align*}
with some $C_1>0$ and $d_1 := \frac{nq}{2q+nq-2n} \in (0,1)$, 
which with Lemma \ref{lem;asy;energy} 
shows that there is $t_1>t_0$ such that 
\begin{align*}
 \lp{\infty}{v\cd - \ol{u_0}} \le C_2 e^{-(1-d_1)\ell t}
 \quad \mbox{for all}\ t>t_1
\end{align*} 
with $C_2:=C_1 K_6 \big( 2K_2 K_3 + |\Omega|^{\frac 1q}\ol{u_0}\big)^{d_1}$. 
On the other hand, we next verify that 
\begin{align}\label{aim;ineq;conv}
 \lp{\infty}{u\cd - \ol{u_0}} \le 
 C e^{-kt} \quad \mbox{for all} \ t> T. 
\end{align}
with some $C,k,T>0$. 
Let $\theta > n$ and let $\alpha \in (\frac n{2r},1)$. 
Letting $A$ again denote the sectorial realization of $-\Delta + \sigma$ 
in $L^\theta (\Omega)$ under homogeneous Neumann boundary conditions,  from Lemma \ref{lem;esti;fracu} 
we can find $t_2> t_1$ such that 
\begin{align}\label{ineq;fracu;inlastproof}
 \lp{\theta}{A^\alpha u\cd} \le 2 K_4 \io u_0 
 \quad \mbox{for all} \ t> t_2. 
\end{align}
Now we fix $\alpha_0 \in (\frac{n}{2\theta},\alpha)$.
Then the embedding $D(A^{\alpha_0}) \hookrightarrow L^\infty(\Omega)$ enables us to find a constant $C_3 >0$ such that 
\begin{align}
 \lp{\infty}{\varphi}\label{ineq;em} 
 \le
 C_3\lp{\theta}{A^{\alpha_0} \varphi} 
 \quad \mbox{for all} \ \varphi\in D(A^{\alpha_0}) 
\end{align}
holds. 
Now noticing from a standard interpolation inequality  
\[ 
 \lp{\theta}{A^{\alpha_0} \varphi} 
 \le 
 C_4\lp{\theta}{A^\alpha \varphi}^{d_2} 
 \lp{\theta}{\varphi}^{1-d_2} 
 \quad \mbox{for all} \ \varphi\in D(A^\alpha) 
\]
with some $C_4>0$ and $d_2:= \frac{\alpha_0}{\alpha} \in (0,1)$, 
and combination with H\"older's inequality and \eqref{ineq;em} 
\[
 \lp{\theta}{\varphi} \le 
 \lp{\infty}{\varphi}^{\frac{\theta-2}{\theta}}
  \lp{2}{\varphi}^{\frac{2}{\theta}} 
 \le C_3^{\frac{\theta-2}{\theta}}\lp{\theta}{A^{\alpha_0} \varphi}^{\frac{\theta-2}{\theta}}\lp{2}{\varphi}^{\frac{2}{\theta}} 
 \quad \mbox{for all} \ \varphi \in D(A^{\alpha_0}),
\] 
we find that 
\begin{align*}
 \lp{\theta}{A^{\alpha_0}\varphi} 
 \le C_5 \lp{\theta}{A^\alpha \varphi}^{\frac{d_2\theta}{d_2\theta +2(1-d_2)}}\lp{2}{\varphi}^{\frac{2(1-d_2)}{d_2\theta + 2(1-\theta)}}
 \quad \mbox{for all} \ \varphi\in D(A^\alpha) 
\end{align*}
with $C_5:= (C_3^{\frac{(\theta-2)(1-d_2)}{\theta}}C_4)^{\frac{\theta}{d_2\theta + 2(1-d_2)}}$. 
Hence, we establish from \eqref{ineq;em} that 
\begin{align*}
 \lp{\infty}{\varphi} 
 \le C_3 C_5\lp{\theta}{A^\alpha \varphi}^{\frac{d_2\theta}{d_2\theta +2(1-d_2)}}\lp{2}{\varphi}^{\frac{2(1-d_2)}{d_2\theta + 2(1-\theta)}}
 \quad \mbox{for all} \ \varphi\in D(A^\alpha). 
\end{align*}
Applying this to $\varphi := u\cd -\ol{u_0}$, 
we can attain from \eqref{ineq;fracu;inlastproof} and 
Lemma \ref{lem;asy;energy} that there is $t_3 > t_2$ such that 
\begin{align*}
 \lp{\infty}{u\cd -\ol{u_0}} &\le C_3 C_5\left(\lp{\theta}{A^\alpha(u\cd-\ol{u_0})}\right)^{\frac{d_2\theta}{d_2\theta +2(1-d_2)}}
 \lp{2}{u\cd-\ol{u_0}}^{\frac{2(1-d_2)}{d_2\theta + 2(1-\theta)}}\\&\le
 C_3 C_5 \left(2 K_4 \io u_0 +\lp{\theta}{A^\alpha \ol{u_0}}\right)^{\frac{d_2\theta}{d_2\theta +2(1-d_2)}}
 K_6^{\frac{2(1-d_2)}{d_2\theta + 2(1-\theta)}} e^{-{\frac{2(1-d_2)\ell t}{d_2\theta + 2(1-\theta)}}}
\end{align*}
for all $t>t_3$, which concludes the proof. 
\end{proof}

Finally, collecting three of the previous results we can establish Theorem \ref{mainthm}. 

\begin{proof}[{\bf Proof of Theorem \ref{mainthm} and Remark \ref{rem;deltaindependence}}] 
Given $M$, $v_\star$, $a$, $k$ as in the theorem, we let $η$ be as in \eqref{def;eta;GL}, $M_0:=M$, choose $χ_0\in(0,k(a+η)^{k-1}\sqrt{\f{2}n})$ and pick $δ_1:=δ(χ_0,a,k,M_0)$ from Lemma \ref{lem;asy;energy}. We define
\begin{equation}\label{definedelta}
 δ:=\min\bigg\{χ_0, k(a+K_1M_0)^{k-1}\sqrt{\f2n},\f{(a+K_1M)^k}M δ_1\bigg\}.
\end{equation}
Then Lemma \ref{lem;GE} is applicable and guarantees global existence of the solution, and Lemma \ref{lem;conv;Linfty} ensures the convergence statement and estimate \eqref{result:convergenceestimate} on $(t_\star,\infty)$. Due to continuity of the solutions, upon proper choice of the constants, the estimate holds on $(0,\infty)$ as claimed in \eqref{result:convergenceestimate}.\\
If $a>0$, we could instead choose $η=0$ and $M_0:=0$ and in \eqref{definedelta} replace $\f{(a+K_1M)^k}M$ by the positive number $\inf \big\{\f{(a+K_1\mu)^k}{\mu}\mid \mu >0 \big\}$ (and the interval $(0,δ)$ would still be nonempty), thus removing any dependence of $δ$ on $M$ and $v_\star$.
\end{proof}

At this point, also the proof of the second theorem follows easily: 
\begin{proof}[{\bf Proof of Theorem \ref{secondthm}}]
We let $M_1$ be so large that $\eta':= 4 \left(1+\sqrt{1+\f{4 \vstar}{c_0 M_1}}\right)^{-2}\vstar$ (cf. \eqref{def;eta;GL}) satisfies 
\[
 k(a+\eta')^{k-1}\sqrt{\f2n} > χ_0
\]
and that $k(a+K_1M_1)^{k-1}\sqrt{\f2n}>χ_0$. Again taking $δ_1:=δ(χ_0,a,k,M_1)$ from Lemma \ref{lem;asy;energy}, by choosing 
\[
 M_0\ge M_1 \text{ such that } \f{(a+K_1M)^k}{M} > \f{χ_0}{δ_1} \text{ for all } M\ge M_0, 
\]
we can apparently ensure that for $M\ge M_0$ and $χ\in(0,χ_0)$ we have that 
\[
 χ<δ:=\min\bigg\{χ_0, k(a+K_1M_0)^{k-1}\sqrt{\f2n},\f{(a+K_1M)^k}M δ_1\bigg\}
\]
and hence may apply Lemma \ref{lem;GE} and Lemma \ref{lem;conv;Linfty} as in the proof of Theorem \ref{mainthm}. 
\end{proof}

%

\section{Acknowledgement} 
T.B. and J.L. acknowledge support of the {\em Deutsche Forschungsgemeinschaft}  within the project {\em Analysis of chemotactic cross-diffusion in complex frameworks}.
M.M. is funded by JSPS Research 
Fellowships for Young Scientists (No.\ 17J00101) and JSPS Overseas Challenge Program for Young Researchers. 

%


%
{\footnotesize 
\def\cprime{$'$}


}
\end{document}